\documentclass[psamsfonts]{amsart}

\usepackage{amssymb}
\usepackage{amsmath}
\usepackage{amsthm}
\usepackage[urlcolor=black, colorlinks=true, citecolor=black, linkcolor=black]{hyperref}
\usepackage[english]{babel}

\newtheorem{theorem}{Theorem}[section]
\newtheorem{lemma}[theorem]{Lemma}
\newtheorem{corollary}[theorem]{Corollary}
\newtheorem{question}[theorem]{Question}
\newtheorem{example}[theorem]{Example}
\theoremstyle{definition}

\theoremstyle{remark}

\numberwithin{equation}{section}

\newcommand{ \R } { \mathbb{R} }
\newcommand{ \N } { \mathbb{N} }

\newcommand{\Q} {\mathbb{Q}}

\newcommand{\script}{\mathcal}
\newcommand{\parentheses}[1]{{\left( {#1} \right)}}
\newcommand{\p}{\parentheses}
\newcommand{\of}{\parentheses}
\newcommand{\singletonDeletion}[1]{\script{D}\parentheses{#1}}
\newcommand{\multideck}[1]{\script{D}'\parentheses{#1}}
\newcommand{\closure}[1]{\overline{#1}}
\newcommand{\closureIn}[2]{\closure{#2}^{#1}}
\newcommand{\interior}[1]{\mathrm{int}\of{#1}}
\newcommand{\Set}[1]{{\left\lbrace {#1} \right\rbrace}}
\newcommand{\singleton}{\Set}
\newcommand{\union}{\cup}
\newcommand{\intersect}{\cap}
\newcommand{\Union}{\bigcup}
\newcommand{\disjointSum}{\oplus}

\newcommand{\cardinality}[1]{{\left\lvert {#1} \right\rvert}}

\def\set#1:#2{\Set{{#1} \colon {#2}}}

\newcommand{\discrete}[1]{\mathfrak{D}_{{#1}}}
\newcommand{\Cech}{\v{Cech}}

\begin{document}
\title{A Topological Variation of the Reconstruction Conjecture}
\author{Max F. Pitz}
\address{Mathematical Institute\\University of Oxford\\Oxford OX2 6GG\\United Kingdom}
\email[Corresponding author]{pitz@maths.ox.ac.uk}
\author{Rolf Suabedissen}
\address{Mathematical Institute\\University of Oxford\\Oxford OX2 6GG\\United Kingdom}
\email{suabedis@maths.ox.ac.uk}

\subjclass[2010]{Primary 54B99; Secondary 05C60, 54B05, 54A05, 54Dxx, 54E35}
\keywords{Reconstruction conjecture, topological reconstruction, singleton deletions, subspace reconstruction}


\begin{abstract}
This paper investigates topological reconstruction, related to the reconstruction conjecture in graph theory. We ask whether the homeomorphism types of subspaces of a space $X$ which are obtained by deleting singletons determine $X$ uniquely up to homeomorphism. If the question can be answered affirmatively, such a space is called reconstructible.	

We prove that in various cases topological properties can be reconstructed. As main result we find that familiar spaces such as 
the reals $\R$, the rationals $\Q$ and the irrationals $P$ are reconstructible, as well as spaces occurring as Stone-\v{C}ech compactifications. Moreover, some non-reconstructible spaces are discovered, amongst them the Cantor set $C$.
\end{abstract}

\maketitle

\section{Introduction}
In 1941, S.M. Ulam and P.J. Kelly proposed a conjecture which has been known since then as the \emph{reconstruction conjecture}. Roughly speaking, the reconstruction conjecture asks whether every finite graph is uniquely determined by the structure of subgraphs which are obtained by deleting a single vertex and all incident edges. Information about the reconstruction conjecture can be found in the survey by J.A. Bondy, \cite{Bondy}. The conjecture remains unsolved -- and is considered as one of the most challenging problems in graph theory.

In this paper we describe a topological version of the reconstruction problem: when do certain subspaces of topological spaces determine the space uniquely up to homeomorphism? Embracing terminology from graph theory, we say that a topological space $Y$ is a \emph{card} of another space $X$ if $Y$ is homeomorphic to $X \setminus \singleton{x}$ for some $x$ in $X$. The \emph{deck} of a space $X$ is a transversal for the non-homeomorphic cards of $X$, i.e.\ an object recording the topologically distinct subspaces one can obtain by deleting singletons from $X$.

Formally, for a space $X$ we denote by $[X]_\sim$ the homeomorphism type of $X$. The deck of $X$ can then be defined as the set $\singletonDeletion{X}=\set{[X \setminus \singleton{x}]_\sim}:{x \in X}$, and cards of $X$ correspond to elements of the deck of $X$. Note that we deliberately identify a card, which is a concrete topological space $Y$, with the class of spaces homeomorphic to $Y$. In practice this identification never causes confusion, and, as in the following examples, we will simply state which cards occur.

If $Y$ is a card of the real line $\R$, then $Y$ is homeomorphic to two copies of the real line. We write this as $\singletonDeletion{\R}=\Set{\R \oplus \R}$. Similarly $\singletonDeletion{\R^n}=\Set{\R^n \setminus \singleton{0}}$. In the case of the unit interval $I$ we have $\singletonDeletion{I}=\Set{[0,1), [0,1) \oplus [0,1)}$. For the sphere, stereographic projection gives $\singletonDeletion{S^n}=\Set{\R^n}$. The Cantor set has the deck $\singletonDeletion{C}=\singleton{C \setminus \singleton{0}}$. Lastly, Cantor's back-and-forth method gives $\singletonDeletion{\Q} = \Set{\Q}$ and $\singletonDeletion{P}=\Set{P}$ for the rationals $\Q$ and the irrationals $P$.
 
We now introduce the central concept of this paper, the notion of reconstruction. Given topological spaces $X$ and $Z$, we say that $Z$ is a \emph{reconstruction of} $X$ if their decks agree. A topological space $X$ is said to be \emph{reconstructible} if the only reconstructions of it are the spaces homeomorphic to $X$. In the same spirit, we say that a property of topological spaces is reconstructible if it is preserved under reconstruction.

Formally, a space $X$ is reconstructible if $\singletonDeletion{X}=\singletonDeletion{Z}$ implies $X \cong Z$ and a property $\script{P}$ of topological spaces is reconstructible if $\singletonDeletion{X}=\singletonDeletion{Z}$ implies ``$X$ has $\script{P}$ if and only if $Z$ has $\script{P}$".

We remark that one could just as well consider the deck of a space as multi-set instead of just a set -- we define the \emph{multi-deck} of a space $X$ to be the multi-set $\multideck{X}=\set{[X \setminus \singleton{x}]_\sim}:{x \in X}$. In other words, the multi-deck not only knows which cards occur, but also how often they occur. If a space is reconstructible from its multi-deck, we will say it is \emph{weakly reconstructible}. The formal definitions are exactly the same as above, where $\singletonDeletion{X}$ is replaced by $\multideck{X}$. Clearly, reconstructible spaces and properties are weakly reconstructible.


In Section \ref{section:counterexamples} we give examples of some reconstructible spaces and present different techniques for reconstructing topological spaces. We also recall some classic topological characterisations of common topological spaces. Further, we present examples of non-reconstructible spaces and comment on which properties of them are not reconstructible, most importantly that neither compactness nor connectedness are reconstructible in general. 

These examples inform and delimit the subsequent investigation into which properties are reconstructible. The recurring theme will be that for a space $X$ having a topological property $\mathcal{P}$, if sufficiently many cards share $\mathcal{P}$ then every reconstruction will satisfy $\mathcal{P}$. Thus, in Section \ref{section:separation} we will show that the common separation axioms with the exception of normality are reconstructible. We will also discuss normality in some detail in that section. We will end this section with a short proof that in $T_1$-spaces local properties are reconstructible. In Section \ref{section:cardinalInvariants} we will show that cardinal invariants such as weight, character, density and spread are reconstructible. We will also investigate the number of isolated points of a space and show that spaces with a finite number of isolated points are reconstructible.

With these tools we show in Section \ref{section:StoneCech} that in Hausdorff spaces, reconstructing compactness is equivalent to reconstructing the space. We prove that Stone-\Cech\ spaces (Stone-\Cech\ compactifications of non-compact Tychonoff spaces) and spaces that arise as maximal finite-point compactifications are reconstructible. We will also use these ideas to show that a variety of topological properties which can be expressed in terms of the Stone-\Cech\ compactification are reconstructible.

Next, in Section \ref{section:compactness} we take a brief look at compactness-like properties. By considering the weaker properties Lindel\"ofness and pseudocompactness we can show that if a compact Hausdorff space $X$ contains a $G_\delta$-point and a non-$G_\delta$-point then it is reconstructible. We finish this section by showing that metrizability and complete metrizability is reconstructible.

In the final Section \ref{section:connectedness} we consider connectedness: although not reconstructible in general, we manage to show that connectedness is weakly reconstructible in separable Tychonoff spaces. The main interest here lies in the fact that there are connected separable Tychonoff spaces without connected cards. It is also remarkable, that this is the only example where our reconstruction result depends on the deck being a multi-set. For all other results, it is not important to know how often a particular card occurs.

Our terminology follows \cite{Engelking}, except that we do no automatically include separation properties in the definition of compactness properties. Unless otherwise indicated, higher separation axioms always include lower separation axioms. The disjoint sum of topological spaces $X_i, i \in \kappa$ will be denoted by $\disjointSum_{i \in \kappa} X_i$. If all $X_i$ are homeomorphic to $X$, we also write this as $\kappa \cdot X$. For a cardinal number $\kappa$, we denote by $\mathfrak{I}_{\kappa}$ the indiscrete space and by $\discrete{\kappa}$ the discrete space on $\kappa$ elements.

\section{Examples and Counterexamples}\label{section:counterexamples}

\subsection{Examples of reconstructible spaces}
One important tool in showing that certain spaces are reconstructible is to characterise them and then show that each of the characterising properties is reconstructible. We will list here some of the spaces we show are reconstructible with forward pointers to the theorems that are needed to complete the proofs.

\begin{theorem}\label{thm:CantorCube}
For uncountable $\kappa$, the Cantor cube $2^\kappa$ and $[0,1]^\kappa$ is reconstructible. None of their cards are compact, paracompact, Lindel\"of or normal.
\end{theorem}
\begin{proof}
For uncountable $\kappa$ and $X=2$ or $[0,1]$, the space $X^\kappa$ is a Stone-\Cech{} compactification of a non-compact spaces \cite{Glicksberg} and hence by Theorem \ref{thm:StoneCech} reconstructible. That none of its cards have the properties from above follows from the fact that $X^\kappa$ is homogeneous and contains the Tychonoff plank $(\omega_1+1) \times (\omega+1)$.
\end{proof}

\begin{theorem}
\label{mythm15}
\label{thm:reals}\label{thm:rationals}\label{thm:irrationals}
 The space of real numbers $\R$, the space of rational numbers $\Q$ and the space of irrational numbers $P$ are  reconstructible spaces.
 \end{theorem}

\begin{proof} The reals: Ward proved in \cite{Ward} that every metric, locally connected, separable connected space, where every card has exactly two components, is homeomorphic to the space of real numbers. By Theorems \ref{mythm14}, \ref{localproperty} and \ref{mythm7}, any reconstruction of $\R$ is metrizable, locally connected and separable. Moreover, Corollary \ref{mycor1} gives that any reconstruction of $\R$ is connected. 
Thus, $\R$ is reconstructible.

The rationals: By Sierpinski's theorem (see \cite{Sierpinski} or \cite[1.9.6]{Mill01}), every dense-in-itself countable metrizable space is homeomorphic to the space of rational numbers. By Theorems \ref{mythm18} and \ref{mythm14}, any reconstruction of $\Q$ shares these properties. Hence $\Q$ is reconstructible.
	
The irrationals: Alexandroff and Urysohn (see \cite{Alexandroff} or \cite[1.9.8]{Mill01}) showed that every topologically complete, zero-dimensional, separable metric space which does not contain a proper open compact subset is homeomorphic to the space $P$ of irrational numbers. If $X$ is a reconstruction of $P$ then by Theorems \ref{completerec}, \ref{localproperty} and \ref{mythm7}, $X$ is a completely metrizable, zero-dimensional, separable space. We will show that $X$ contains no proper open compact subset. Let $U \subset X$ be an open subset. If $U \subsetneq X$, there exists an $x \in X$ such that $U \subset X \setminus \singleton{x} \cong P$. This shows that $U$ is not compact. Now assume that $U=X$ is a compact space. Then every card would inherit local compactness. This contradicts the fact that $P$ is not locally compact. Thus, $X$ contains no proper open compact subset and therefore $X \cong P$.
\end{proof} 

\begin{theorem}
\label{mythm24}
	The unit interval $I$ and the spheres $S^n$ are reconstructible.
\end{theorem}

\begin{lemma}[{\cite{Magill} or \cite[6.10 \& 6.11]{Chandler}}]
\label{mythm23}
(1) The only finite compactifications of $\R$ are $S^1$ and $[-\infty,\infty]$. (2) The only finite compactification of $\R^n$ $(n \geq 2)$ is $S^n$. \qed \end{lemma}

\begin{proof}
$\mathbf{I}$. $I$ has a card $Y=[0,1)$. Now any 2-point compactification of $Y$ would be a 3-point compactification of $(0,1)$, contradicting Lemma \ref{mythm23}(1). Applying Theorem \ref{mythm21} gives that $I$ is reconstructible. 

$\mathbf{S^1}$. Suppose $X$ is a reconstruction of $\singletonDeletion{S^1}=\Set{\R}$ which is not compact. As $X$ must be locally compact Hausdorff, there exists an Alexandroff compactification $\omega X$ which then is a 2-point compactification of $Y=\R$. Lemma \ref{mythm23}(1) tells us that all such compactifications are unique, thus $\omega X \cong [0,1]$. Since $X$ must be connected, we have $X \cong [0,1)$. But then $\singletonDeletion{X} \neq \Set{\R}$. Hence $X$ must be compact and therefore is the unique minimal compactification of $\R$, which is $S^1$. 

$\mathbf{S^n}, \; n \geq 2$. Lemma \ref{mythm23}(2) gives that no card has a 2-point compactification, so we may apply Theorem \ref{mythm21}. \end{proof}

Note that a similar method can be used in order to show that the Euclidean spaces $\R^n$ for $n \geq 2$ are reconstructible. In fact, one shows that $\R^n \setminus \singleton{0}$ has only a 1-point compactification or the 2-point compactification $S^n$. Then it is not hard to see that if $X$ is a reconstruction from $\singletonDeletion{\R^n}$ then $X$ cannot be compact and $\omega X \cong S^n$.

\subsection{Non-reconstructible spaces and properties}
In the graph-theoretic case, the standard infinite counterexample to the reconstruction conjecture is the tree of countable infinite degree $T$, as its deck $\script{D}(T)=\Set{{\aleph_0 \cdot T}}$ is the same as of say $\aleph_0 \cdot T$ itself. In the topological world, this has a surprising analogy, namely the Cantor set C.

\begin{example}
We have $\singletonDeletion{C}=\singletonDeletion{C \setminus \singleton{0}}$, hence the Cantor set is not reconstructible. In particular, the properties of compactness, countable compactness and pseudocompactness are non-reconstructible.
\end{example}

\begin{proof}
This follows easily from the observation that $C \setminus \singleton{0} \cong \aleph_0 \cdot C$.
\end{proof}

Interestingly, the Cantor set is just one particular example of a whole family of non-reconstructible spaces. Spaces which only have $\lambda$ different homeomorphism types amongst their open subspaces (for some cardinal $\lambda$) are said to be of \emph{diversity} $\lambda$ \cite{diversity}. The Cantor set is a compact Hausdorff space of diversity 2, and it is easy to see that in fact every such space $X$ is non-reconstructible, as $\singletonDeletion{X}=\singletonDeletion{X \setminus \singleton{x}}$. Hence, for example, the Double Arrow space $D$ and also the product $D \times C$ are non-reconstructible \cite{diversity}.


\begin{example}
Connectedness is a non-reconstructible property. Also, there are spaces with an arbitrary large number of non-homeomorphic reconstructions.
\end{example}

We use the same idea as in the previous example: construct a space where every card is an infinite disjoint sum of the original space.

\begin{proof}[Construction] We present a sketch of the construction. Start with the open unit interval and replace every point by the open hedgehog of spininess $\kappa$ with the centre taking the place of the original point. We continue inductively, replacing every non-central point again by a hedgehog of spininess $\kappa$ (with half-open spines). If we put the natural metric onto the resulting space, namely, the distance between two points is the \lq shortest walk\rq\ from one to the other along the spines, we obtain a connected metrizable space $X$ with cards homeomorphic to $\kappa \cdot X$ (removing a point leaves the components: everything to the left along the spine containing the point; everything to the right along that spine; and everything on each of the spines starting at that point), showing that connectedness is not reconstructible. 

For $\lambda \le \kappa$, every card of $\lambda \cdot X$ is homeomorphic to $\kappa \cdot X$. Therefore, $X$ has $\kappa$ many distinct reconstructions.
\end{proof}

\begin{example}
Lindel\"ofness is a non-reconstructible property.
\end{example}

\begin{proof}[Construction] We use the following, somewhat different topology, inspired by filtration spaces \cite{Juhasz} and resolutions \cite{Watson}.

Fix a cardinal $\kappa$. For $n \in \omega$, let $L_n = (0,1) \times (\kappa \times (0,1))^n$ and set $X = \Union_n L_n$.  Viewing elements of $X$ as partial functions from $\omega$, the notion of extension makes sense. For $x=(x_0,\lambda_1,x_1,\dots,\lambda_n,x_n) \in X$ and $\epsilon > 0$ we write 
\begin{eqnarray*}
(x-\epsilon,x+\epsilon) &=& \{y \in X\colon y \text{ equals or extends } (x_0,\dots,\lambda_n,a), \\
&& \text{ where } a \in (0,1) \intersect (x_n-\epsilon,x_n+\epsilon)\} 
\end{eqnarray*}
and similarly for $[x-\epsilon,x+\epsilon]$. For $x \in X$, $y_1,\dots,y_n$ extending $x$ and $\epsilon,\delta >0$ define 
$$
\langle x,\epsilon,y_1,\delta,\dots,y_n,\delta \rangle = (x-\epsilon,x+\epsilon) \setminus \Union [y_j-\delta,y_j+\delta].
$$
Sets of this form are a neighbourhood basis at  $x \in X$. It is Hausdorff, regular and since each $L_n$ is Lindel\"of, so is $X$. Every card is a disjoint sum of $\kappa$-many copies of $X$ and hence every card is in fact also a reconstruction of $X$. But $X$ is connected and Lindel\"of whereas the cards are not (if $\kappa$ is uncountable).
\end{proof}

Our last example, which was pointed out to the authors by Mika G\"o\"os, shows that Theorem \ref{mythm19}, which states that $T_1$ spaces with finitely many isolated points are reconstructible, does not hold in the infinite case.

\begin{example}
There exists a non-reconstructible space containing infinitely many isolated points.
\end{example}

\begin{proof}[Construction]
Denote by $O$ the hedgehog of countable spininess (again, with half-open spines). Let
\[X_1=O--O==O--O==O--O== \cdots\]
\[X_2=O==O--O==O--O==O-- \cdots\] where $O--O$ means connecting the centres with a simple copy of $[0,1]$ and $O==O$ means connecting the centres with countably many copies of $[0,1]$. Now for $n \in \N$ let $G_i^n$, $i=1,2$, be the initial sequences of the spaces $X_i$ ending at the $n^{th}$ centre, e.g.\ $G_1^4=O--O==O--O$. Finally, define $G$ to be 
\[G= (\aleph_0 \cdot (0,1)) \oplus  \bigoplus_{n\geq1} (\aleph_0 \cdot G_1^n)  \oplus \bigoplus_{n\geq1} (\aleph_0 \cdot G_2^n).\] Then $\singletonDeletion{X_1 \oplus G}= \Set{X_1 \oplus G,X_2 \oplus G}=\singletonDeletion{X_2 \oplus G}$, hence both spaces are non-reconstructible. Similarly, the spaces $X_1 \oplus G \oplus \discrete{\aleph_0}$ and $X_2 \oplus G \oplus \discrete{\aleph_0}$ have the same deck and hence are non-reconstructible.
\end{proof}

\section{Separation Properties}\label{section:separation}
In this section we will see that all hereditary separation axioms are reconstructible. 
		
\begin{theorem}
\label{mythm8}
If $\cardinality{X} \geq 3$, then the property $T_i$ is reconstructible for $i \neq 4$. For these $i$, $X$ is $T_i$ if and only if every card in the deck $\singletonDeletion{X}$ is $T_i$. In the $T_4$ case we have that if all cards $Y \in \singletonDeletion{X}$ are $T_3$ and one of them is $T_4$, then $X$ is $T_4$ as well.
\end{theorem}

\begin{proof} For hereditary axioms $T_i$ ($i \neq 4$) it suffices to prove the converse direction. 

$\mathbf{T_0}$. Let $x \neq y \in X$ and find $z \in X \setminus \Set{x,y}$. Then $x,y \in Y=X \setminus \singleton{z}$ and since $Y$ is $T_0$ there exists an open set $U \subset X$ such that without loss of generality $x \in U \cap Y$ and $y \notin U \cap Y$. In $X$, we consequently have $x \in U$ and $y \notin U$, i.e.\ $X$ is $T_0$.

$\mathbf{T_1}$. We show that for every $x \in X$ the singleton $\singleton{x}$ is closed. Find $ \Set{x, y, z} \in X$. Since both $X \setminus \singleton{y}$ and $X \setminus \singleton{z}$ are $T_1$-spaces by assumption, $\singleton{x}$ is closed in these spaces. Now if $\singleton{x}$ was not closed in $X$, then by definition of the subspace topology the sets $\Set{x,y}$ and $\Set{x,z}$ must be closed in $X$. Thus, $\singleton{x}=\Set{x,y} \cap \Set{x,z}$ is closed, a contradiction.

$\mathbf{T_2}$. Let $x \neq y \in X$ and find $z \in X \setminus \Set{x,y}$. Then $x,y \in Y=X \setminus \singleton{z}$ and since $Y$ is $T_2$ there exist open sets $U,V \subset Y$ such that $x \in U$, $y \in V$ and $U \cap V = \emptyset$. Supposing that every card in our deck is $T_2$ we know by the previous part that $X$ is $T_1$ and hence that $Y$ must be an open subspace. Thus, $U$ and $V$ are open in $X$ as well and therefore $X$ is $T_2$. 


$\mathbf{T_3}$. Let $A \subset X$ be a closed set and let $x \in X \setminus A$. Since the property $T_2$ is reconstructible, it suffices to separate $x$ from $A$ by open sets. Moreover, we may assume $\cardinality{A} \geq 2$. Let $y \in A$ and $Y=X \setminus \singleton{y}$. Then $ \emptyset \neq A \setminus \singleton{y}= A \cap Y$ is a closed subset of $Y$. By assumption, $Y$ is regular, so there exist disjoint open sets $U,V \subset Y$ such that $A \setminus \singleton{y} \subset U$ and $x \in V$. Because $X$ is Hausdorff, the sets $U$ and $V$ are open in X. Further, we can separate $x$ and $y$ by disjoint open sets $M \ni x$ and $N \ni y$. Thus, we can separate $x$ from $A$ by $V \cap M \ni x$ and $U \cup N \supset A$. 

$\mathbf{T_{3\frac12}}$. Let $A \subset X$ be a closed set and let $x \in X \setminus A$. Since $T_{3\frac12}$ implies $T_3$ we know by the previous part of this proof that $X$ must be $T_3$. Hence we can find an open $U \subset X$ such that $x \in U \subset \closure{U} \subset (X \setminus A)$. Fix $y \in X \setminus \closure{U}$. By assumption, we can find a continuous mapping \[f \colon X \setminus \singleton{y} \to [0,1] \; \text{ such that } \; f(x)= 1\; \text{ and } \; f((X \setminus U) \cap (X \setminus \singleton{y})) = 0.\] The functions $g \colon X \setminus U \to \singleton{0}$ and $f|_{\closure{U}} \colon \closure{U} \to [0,1]$ are both continuous, defined on closed subsets of $X$ and coincide on their intersection. By the pasting lemma, the continuous function $h \colon X \to [0,1]$ coinciding with $f$ and $g$ where defined then separates the pair $(A,x)$.

$\mathbf{T_4}.$ Suppose all cards in $\singletonDeletion{X}$ are $T_3$ and $Y=X \setminus \singleton{x}$ is $T_4$. The previous parts show that $X$ is $T_3$. If $X$ was not normal, there existed closed disjoint subsets $A,B \subset X$ which we cannot separate by open sets. If $x \notin A \cup B$, then we can separate $A$ from $B$ by open, disjoint $U,V \subset Y$. As $Y$ is an open subspace, $U$ and $V$ are open in $X$, separating $A$ from $B$. This is impossible, hence without loss of generality we may assume that $x \in A$. This time, we find disjoint open $U,V \subset Y$ separating $A \setminus \singleton{x}$ from $B$. Again, $U$ and $V$ are $X$-open as well. Further, since $X$ is $T_3$, we can find an open $W \subset X$ such that $x \in W \subset \closure{W} \subset X \setminus B$. But now $U \cup W$ and $V \setminus \closure{W}$ separate $A$ and $B$, a contradiction.

$\mathbf{T_{5}}$. Assume that every card in $\singletonDeletion{X}$ is $T_5$, but $X$ is not. Then there exists a subspace $Y \subset X$ which is not $T_4$. The previous part guarantees that $X$ is $T_4$, thus $Y \subsetneq X$. With $x \in X \setminus Y$ the card corresponding to $X \setminus \singleton{x}$ contains the non-$T_4$ subspace $Y$, a contradiction. 

$\mathbf{T_{6}}$. Assume that every card in $\singletonDeletion{X}$ is $T_6$, but $X$ is not. We may assume that $X$ is $T_1$ but there exists some closed subset $A \subsetneq X$ which fails to be a $G_{\delta}$-set. Choose $x \in X \setminus A$, then $Y=X \setminus \singleton{x}$ is $T_6$, so there exists a countable family of open sets $\script{V}$ in $Y$ such that $A = \bigcap{\script{V}}$. But because $X$ was $T_1$, all $V \in \script{V}$ are $X$-open, a contradiction.  
\end{proof}

We remark that our reconstruction result about $T_4$-spaces is sufficient but not necessary. As noted in Theorem \ref{thm:CantorCube} there are compact Hausdorff spaces such that all cards are non-normal. Even more, it is consistent with the usual axioms of set theory ZFC that normality is not reconstructible: we have an example, requiring the Continuum Hypothesis, where normality is non-reconstructible. However, we do not yet know of a ZFC example of a normal space with a non-normal reconstruction.

\begin{question}
Is there a ZFC example showing that normality is non-reconstructible? And under what additional assumptions is normality reconstructible (e.g. is every reconstruction of a separable normal space normal)?
\end{question}

When investigating other topological properties, it is often useful to assume some of the lower separation axioms. Some authors even include those in the definition of the relevant property. If this allows us to obtain certain results, we will freely assume separation axioms up to and including complete regularity. As an example we can easily show that \lq local\rq\ topological properties are reconstructible in $T_1$-spaces:

For a topological property $\mathcal{P}$, we will say that a space $X$ is \lq locally $\mathcal{P}$\rq\ and call \lq locally $\mathcal{P}$\rq\ a local property if and only if every neighbourhood filter has a basis of sets which satisfy $\mathcal{P}$. In other words, $X$ is locally $\mathcal{P}$ if and only if for every $x$ and open $U \ni x$ there is $A \subseteq X$ such that $x \in \interior{A} \subseteq A \subseteq U$ and $A$ is $\mathcal{P}$. Examples of local properties are locally compact, locally connected, locally metrizable (we note that every locally metrizable space is $T_1$) but also zero-dimensionality (in the presence of regularity locally zero-dimensional is equivalent to zero-dimensional and every $T_1$ zero-dimensional space is clearly regular).

\begin{theorem}
\label{localproperty}
In the realm of $T_1$-spaces, local properties are reconstructible.
\end{theorem}
\begin{proof} 
We prove that for a topological property $\mathcal{P}$, $X$ is locally $\mathcal{P}$ if and only if all cards are locally $\mathcal{P}$. For the direct implication, note that a local property is hereditary with respect to open subspaces. Conversely, take an arbitrary $x \in X$ and choose $y \neq x$. Now note that a neighbourhood basis of $x$ in the card $X \setminus \Set{y}$ is a neighbourhood basis of $x$ in $X$.
\end{proof}

\section{Cardinal invariants}\label{section:cardinalInvariants}
 It is clear that we can reconstruct the size of a space from its deck, as every card has exactly one element less than the original space. Other cardinal invariants are reconstructible as well. We focus on those we need for other results. Following \cite{Engelking}, we assume all cardinal invariants other than the size of a space to have infinite values only.

\begin{theorem}
\label{mythm4}
The weight $w(X)$ of a topological space $X$ is reconstructible. 
\end{theorem}
\begin{proof}
We claim that $w\p{X}= \sup\set{w \p{Y}}:{Y \in \singletonDeletion{X}}$.

``$\geq$": This follows from the definition of the subspace topology.

``$\leq$": 
We show that if $w\p{Y} < \mathfrak{m}$ for all $Y \in \singletonDeletion{X}$ then $w\p{X} < \mathfrak{m}$. Choose $x_1 \neq x_2 \in X$ and let $Y_i=X \setminus \singleton{x_i}$. By assumption, there exist bases $\script{B}_i$ of $Y_i$ of cardinality less than $ \mathfrak{m}$. 
For $U \in \script{B}_i$ we define $\tilde{U}$ to be the open set it comes from, i.e. $\tilde{U}=U$ if $U$ was $X$-open and $\tilde{U}=U \cup \singleton{x_i}$ otherwise. Let $\tilde{\script{B}}_i=\set{\tilde{U}}:{U \in \script{B}_i}$, then $\script{B}=\tilde{\script{B}}_1 \cup \tilde{\script{B}}_2$ is a base of $X$ with $\cardinality{\script{B}}<\mathfrak{m}$.
\end{proof}

\begin{theorem}
\label{mythm7}
The density $d(X)$ of a topological space $X$ is reconstructible.
\end{theorem}
\begin{proof}
We claim that $d(X)= \inf\set{d(Y)}:{Y \in \singletonDeletion{X}}$.

``$\leq$": Let $Y \in \singletonDeletion{X}$ i.e. $Y= X \setminus \singleton{x}$ for some $x \in X$ and let $A \subset Y$ be a dense subset, i.e. $\closureIn{Y}{A}=Y$. We then have $\closureIn{X}{A \cup \singleton{x}}=X$, yielding $d(X) \leq d(Y)$.

``$\geq$": Let $A \subset X$ be a dense subset. We may assume that $A \neq X$ as otherwise the claim is immediate. Now choose $x \notin A$ and let $Y = X \setminus \singleton{x}$. But then $\closureIn{Y}{A} = \closureIn{X}{A} \intersect Y = Y$.
\end{proof}

With similar methods one can, for example, show that the character, the cellularity and the spread of a space are reconstructible. 

We now prove two theorems about the number of isolated points. We will use them frequently in the following.
\begin{theorem}
\label{mythm18}
In the realm of $T_1$-spaces, the number $i(X)$ of isolated points is reconstructible. In particular, if no card of a space $X$ has an isolated point then $X$ has none either.
\end{theorem}

\begin{proof}
Since $X$ is a $T_1$-space, any isolated point of a card is isolated in $ X$, too. It follows $i(Y) \leq i (X)$. Moreover, since we delete at most one point, we have that $i(X)-1 \leq i(Y) \leq i(X)$ for all $Y \in \singletonDeletion{X}$. If we are dealing with infinitely many isolated points, we are clearly finished. Supposing that all cards have $i(Y) < \infty$, we encounter two cases. If there are $Y_1,Y_2 \in \singletonDeletion{X}$ with $i(Y_1) \neq i(Y_2)$ then $i(X)=\max\Set{i(Y)}$. Otherwise, if for all $Y \in \singletonDeletion{X}$ we have $i(Y)=n$ then it is not hard to see that $i(X) = 0 \text{ if } n=0$ and $i(X)=n+1 \text{ if } n>0$.
\end{proof}

\begin{theorem}
\label{mythm19}
Suppose that $X$ is a $T_1$-space with $\cardinality{X} \geq 3$ and a positive but finite number of isolated points. Then $X$ is reconstructible.
\end{theorem}

\begin{proof}
Let $X$ be a $T_1$-space with finitely many isolated points and let $Z$ be a reconstruction of $X$. There exists a card $Y \in \singletonDeletion{X}$ with $i(X)-1=i(Y)=i(Z)-1$. But this means that the card $Y$ was obtained by deleting an isolated point. Hence $X \cong Y \oplus \discrete{1} \cong Z$, showing that $X$ is reconstructible.
\end{proof}

\section{Classes of reconstructible spaces}\label{section:StoneCech}
We start with an observation which proves that compact Hausdorff spaces containing an isolated point are reconstructible. As an immediate corollary we obtain that spaces such as the Tychonov plank $T$ or the converging sequence $\omega +1$ are reconstructible.

\begin{theorem}
\label{mythm20}
Suppose $X$ is a Hausdorff space with one compact card $Y$. Then $X$ is reconstructible.
\end{theorem}

\begin{proof}
	Let $Y = X \setminus \singleton{x}$ be a compact subspace of $X$. Because in Hausdorff spaces compact subspaces are closed, we know that $Y$ was obtained by deleting an isolated point. Thus $X \cong Y \oplus \discrete{1}$ is reconstructible.
\end{proof}

\begin{theorem}
\label{compactreconstruction} \label{thm:allReconstructionsCompact}
If $X$ is a compact $T_2$-space such that every reconstruction of $X$ is compact then $X$ is reconstructible.
\end{theorem}
\begin{proof}
If $X$ contains an isolated point, the claim follows from Theorem \ref{mythm20}. In the other case, the claim follows from the fact that every reconstruction must be the 1-point compactifications of some (and in fact every) card.
\end{proof}

This theorem explains why reconstructing compactness is so hard: in Hausdorff spaces reconstructing compactness is in fact equivalent to reconstructing the space itself.

We now prove a pair of theorems, showing that certain \lq maximal\rq\ compactifications are reconstructible:

\begin{theorem}\label{mythm21}\label{thm:maximalFiniteCompactification}
Suppose $X$ is a compact Hausdorff space. If there is a space $Y$ such that $X$ is the maximal finite-point compactification of $Y$ (i.e. $X$ is a compactification of $Y$ with finite remainder of size $n$ and there is no $n+1$-point compactification of $Y$), then $X$ is reconstructible. In particular, if one card of $X$ does not have a 2-point compactification then $X$ is reconstructible.
\end{theorem}
\begin{proof}
We may assume that $X$ has no isolated points, as otherwise it is reconstructible by Theorem \ref{mythm20}. Now assume that $Y = X \setminus \Set{x_0,\dots,x_n}$ is as above. Observe that $Y'=X \setminus \Set{x_0}$ is locally compact and that any compactification of $Y'$ is a compactification of $Y$ (as $x_1,\dots,x_n$ are not isolated). Hence $Y'$ does not have a 2-point compactification.

Now assume that $Z$ is a non-compact reconstruction from $\singletonDeletion{X}$. Then $Z$ must be locally compact Hausdorff and contain no isolated points. Hence $Z$ has a one-point compactification $\omega Z$. But then $\omega Z$ would be a two-point compactification of $Y'$, a contradiction.

Hence every reconstruction of $X$ is compact Hausdorff and by Theorem \ref{compactreconstruction} the result follows.
\end{proof}

We call a space $Z$ a \emph{Stone-\v{C}ech space} if it occurs as a Stone-\v{C}ech compactification, meaning there exists a non-compact Tychonov space $X$ such that $Z=\beta X$. A similar idea as above allows us to prove that Stone-\Cech\ spaces are reconstructible:

\begin{theorem}
\label{Stonecech} \label{thm:StoneCech}
Every Stone-\v{C}ech space $Z=\beta X$ is reconstructible.
\end{theorem}
\begin{proof}
If $Z$ contains an isolated point $z$ by Theorem \ref{mythm20}, $Z$ is reconstructible. So we may assume that $Z$ contains no isolated points. Once we have shown that one card has no 2-point compactification, our claim follows from Theorem \ref{mythm21}. Choose $z \in Z \setminus X$. Then $Y=Z \setminus \singleton{z}$ lies in between $X$ and its Stone-\v{C}ech compactification, yielding $\beta Y = Z$ \cite[3.6.9]{Engelking}.
\end{proof}

This theorem raises the question whether there is an elegant way that allows us to decide when a compact Hausdorff space is a Stone-\v{C}ech space. To our knowledge, this question has only been answered in special cases. One of these cases has be investigated in \cite[Thm.\ 2]{Glicksberg}. In this paper, I. Glicksberg showed that when $\kappa$ is uncountable, then both the Cantor cube $\Set{0,1}^{\kappa}$ and the Hilbert cube $I^{\kappa}$ are Stone-\v{C}ech spaces. Thus, they are reconstructible.


We now describe a method which deals with the reconstruction of topological properties that can be characterised in terms of the relation between a Tychonov space $X$ and its Stone-\v{C}ech compactification. Suppose $\script{P}$ is a topological property which is hereditary with respect to cards. We call $\script{P}$ \emph{Stone-\v{C}ech characterisable} if there exists a property of subspaces $\script{Q}$ such that \[X \vdash \script{P}\; (X \text{ has } \script{P}) \quad \text{if and only if} \quad (X,\alpha X) \vdash \script{Q}\] for any compactification $\alpha X$ of $X$. The property $\script{Q}$ is \emph{finitely additive} if whenever $(Y_1,Z) \vdash \script{Q}$ and  $(Y_2,Z) \vdash \script{Q}$ then as well  $(Y_1 \cup Y_2,Z) \vdash \script{Q}$.

\begin{lemma}
	\label{betaXreconstruction}
	Let $X$ be a $T_{3\frac12}$-space. Suppose that $\script{P}$ is a Stone-\v{C}ech characterisable property such that the corresponding $\script{Q}$ is finitely additive. Then $\script{P}$ is reconstructible.
\end{lemma}

\begin{proof}
	Since $\script{P}$ is hereditary with respect to cards, all $Y \in \singletonDeletion{X}$ have $\script{P}$. Suppose that $Z$ is a reconstruction of $X$. In the case where $X$ has infinitely many isolated points, $Z$ must have infinitely many isolated points and hence $Z \in \singletonDeletion{Z} = \singletonDeletion{X}$. Thus $Z \vdash \script{P}$. If $X$ has finitely many isolated points, then $X$ is reconstructible by Theorem \ref{mythm19} and again $X \vdash \script{P}$.
	
	So we are allowed to assume that $X$ and all cards $Y$ are dense-in-itself. Choose $x_1 \neq x_2 \in X$ and consider the cards $Y_i = X \setminus \singleton{x_i}$. Since all spaces are dense-in-itself, $\beta X$ is a compactification of both $Y_i$. By assumption, we have $(Y_i, \beta X) \vdash \script{Q}$. As $\script{Q}$ is finitely additive, we have  $(Y_1 \cup Y_2, \beta X) \vdash \script{Q}$. However, $Y_1 \cup Y_2=X$, so finally $X \vdash \script{P}$, completing the proof.
\end{proof}

\begin{theorem}
\label{cechcomplete}
	\v{C}ech-completeness is reconstructible, i.e.\ a $T_{3\frac12}$-space is \v{C}ech-complete if and only if all cards $Y \in \singletonDeletion{X}$ are \v{C}ech-complete.
\end{theorem}

\begin{proof}
	Let $\script{P}$ denote \v{C}ech-completeness. Recall that a space $X$ has $\script{P}$ if $X$ is a $G_{\delta}$-set in every ($\Leftrightarrow$ one) compactification. We show that $\script{P}$ satisfies all assumptions in Lemma \ref{betaXreconstruction}. 	$\script{P}$ is hereditary with respect to cards. If $Y \subset X$ is an open subspace and $X \subset \beta X$ is a $G_{\delta}$-set then it is clear that $Y$ is a $G_{\delta}$-set in $\beta X$ as well. So $Y$ is a $G_{\delta}$-set in $\closure{Y} \subset \beta X$ which is a compactification of $Y$. This shows that $Y$ is \v{C}ech-complete. Moreover, the union of two $G_{\delta}$-sets is again a $G_{\delta}$-set. This shows that $\script{Q}$ is finitely additive, as required. 
	\end{proof}

\section{Compactness-like properties and metrizability} \label{section:compactness}
Compactness is surely one of the strangest properties with respect to reconstruction. We have already seen that compactness itself is not reconstructible. At the same time, compactness-like properties play an essential role in our results about reconstruction. 

It is clear that if one card of a space $X$ is compact (Lindel\"of, countably compact, pseudocompact) then so is $X$. In the case of Lindel\"ofness and pseudocompactness, we can use this observation in certain circumstances. Since Lindel\"ofness is inherited by $F_{\sigma}$-sets, it is reconstructible in spaces having a $G_\delta$-point, e.g. in $T_6$ and first-countable spaces. Further, recall that a space $X$ is pseudocompact if in one compactification $\alpha X$ every closed $G_{\delta}$-set intersects $X$. This yields that pseudocompactness is reconstructible is spaces containing a non-$G_\delta$-point. Together, these two observations shed light on the question under which conditions compactness is reconstructible.

\begin{theorem}
\label{Gdeltareconstruction}
Suppose $X$ is a compact $T_2$-space containing points $x_1$ and $x_2$ such that $x_1$ is a $G_{\delta}$-point and $x_2$ is not. Then $X$ is reconstructible.
\end{theorem}
\begin{proof}
Again we may assume that $X$ does not contain isolated points. The card $Y_1=X \setminus \singleton{x_1}$ is Lindel\"of and normal. Theorem \ref{mythm8} and the previous remarks give that every reconstruction of $X$ is $T_4$, Lindel\"of and pseudocompact. Now use Theorem \ref{compactreconstruction} together with the fact that in $T_4$-spaces, pseudocompactness and countable compactness are equivalent, gives that every reconstruction is countably compact, Lindel\"of, hence compact. 
\end{proof}

Another reason that makes compactness-like properties so interesting in our context is that is can be used to obtain results about metrizability. Our first result in this direction is about paracompactness. It is worth noting this result's symmetry to the partial reconstruction of the $T_4$ axiom in Theorem \ref{mythm8}.

\begin{lemma}
\label{mylem9}
Suppose that every card $Y \in \singletonDeletion{X}$ is $T_3$ and that one card is paracompact. Then $X$ is paracompact.
\end{lemma}
\begin{proof}
We use \cite[5.1.11]{Engelking}: For a regular space $X$ we have the following equivalence: $X$ is paracompact if and only if every open cover has a locally finite refinement consisting of arbitrary sets. For the proof, let $\script{U}$ be an arbitrary open cover. Identify $x \in X$ such that $X \setminus \singleton{x}$ is paracompact. Choose $A \in \script{U}$ such that $x \in A$. Now $X \setminus A$ is a closed subset of a paracompact space $X \setminus \singleton{x}$, hence paracompact. Choose a locally finite open refinement $\script{U}'$. Now let $\script{U}''=\singleton{A} \cup \set{V \cap (X \setminus A)}:{V \in \script{U'}}$. This is a locally finite refinement (consisting of arbitrary sets) of $\script{U}$, as for all $x \in A$, this $A$ is a neighbourhood intersecting $\script{U}''$ only once. For all other points, local finiteness is inherited from $\script{U}'$. Since regularity is reconstructible by Theorem \ref{mythm8}, the characterisation given in the beginning of this proof validates our claim.
\end{proof}

\begin{theorem}
\label{recpara}
	In the realm of $T_3$-spaces containing one singleton $\singleton{x} \subset X$ which is a $G_{\delta}$-set, paracompactness is reconstructible.
\end{theorem}
\begin{proof} Lemma \ref{mylem9} together with the fact that paracompactness is hereditary with respect to $F_\sigma$ sets (\cite[5.1.28]{Engelking}).
\end{proof}


\begin{theorem}
\label{mythm14}
Metrizability is reconstructible. That is, $X$ is metrizable if and only if every card in $\singletonDeletion{X}$ is metrizable.
\end{theorem}
\begin{proof}
	Every reconstruction of a metrizable space $X$ is locally metrizable by Theorem \ref{localproperty}. Further, in a metric space every point is a $G_\delta$. Thus, by Theorem \ref{recpara} every reconstruction of $X$ is paracompact. Now the Smirnov metrization Theorem ($X$ is metrizable if and only if $X$ is paracompact and locally metrizable) \cite[5.4.A]{Engelking} gives that $X$ is metrizable.
\end{proof}

\begin{theorem}
\label{completerec}
	Complete metrizability is reconstructible, i.e.\ $X$ is completely metrizable if and only if all cards $Y \in \singletonDeletion{X}$ are completely metrizable.
\end{theorem}
\begin{proof}
	A space is completely metrizable if and only if it is metrizable and \v{C}ech-complete. So the claim now follows from Theorems \ref{mythm14} and \ref{cechcomplete}.
\end{proof}

\section{Connectedness properties}\label{section:connectedness}
 In this section, we are again facing the difficulty that cards of connected spaces do not have to be connected. As in the case of compactness it is not hard to see that a space with more than three elements is connected if all cards are connected. A bit stronger: if $X$ has no isolated points then one connected card guarantees that $X$ is connected. Still, we already saw that connectedness is non-reconstructible. On the other side, various kinds of disconnectedness are reconstructible. These results base strongly on the following lemma by Kline.


\begin{lemma}[\cite{Kline}]
\label{mylem3}
	A connected space $X$ with $\cardinality{X} \geq 3$ cannot have more than one dispersion point. Moreover, if one card of $X$ is totally disconnected, then all other cards are connected. \qed
\end{lemma}

\begin{theorem}
\label{mythm10}
Let $\cardinality{X} \geq 3$. Then total disconnectedness is reconstructible, i.e. $X$ is totally disconnected if and only if every card is totally disconnected.
\end{theorem}
\begin{proof}
The direct implication follows from the fact that total disconnectedness is hereditary. For the converse we prove the contraposition. Suppose that $X$ is not totally disconnected. If one component $A \subset X$ has $\cardinality{A} \geq 3$, we can apply Lemma \ref{mylem3} and see that for some $a \in A$, the card corresponding to $X \setminus \singleton{a}$ is not totally disconnected. As a result we may assume that for all components $A_s \subset X$ we have $\cardinality{A_s} \leq 2$ and further that there is at least one component $A \subset X$ with $\cardinality{A}=2$. Suppose there exists a component $B=\singleton{x}$. Then $\singleton{x}$ is a closed set. We claim that $A$ is a component of $Y=X \setminus \singleton{x}$. Otherwise, there exist open disjoint sets $U,V \subset Y$ disconnecting $A$. But since $Y$ was an open subspace, $U$ and $V$ disconnect $A$ in $X$, a contradiction. This leaves us with the possibility that all components $A_s \subset X$ consist of two points. Note that by assumption on the size of $X$, there are at least two components $A_1=\Set{x_1,y_1}$ and $A_2= \Set{x_2,y_2}$. If $A_1$ was not connected in $X \setminus \singleton{x_2} $ then it follows that $\singleton{x_2}$ is open in $X$. Similarly, we have that $\singleton{y_2}$ is open in $X$. Thus, $A_2$ is disconnected. This contradiction proves the claim.
\end{proof}

Although we have seen that in general connectedness is not reconstructible and in fact may not be inherited by any of the cards, traces of connectedness remain in the cards: removing a point will, with few exceptions, not leave behind a very disconnected space. We can use a precise version of this together with the fact that connected Tychonoff spaces have cardinality at least continuum to reconstruct connectedness in certain circumstances.

We note that this is the only proof of this paper which needs the deck to be a multiset instead of a set. It is also the only property which we have shown to be reconstructible without any card having it.

The following is a slight generalisation of a theorem of Whyburn about compact connected metrizable spaces in \cite{Whyburn}. The proof is a modern, slightly simplified and adapted version of Whyburn's original proof.
\begin{theorem}\label{thm:Whyburn}
If $X$ is a $T_1$ connected separable space then for all but countably many points $x$ of $X$, the subspace $X \setminus \Set{x}$ has at most two components.
\end{theorem}
\begin{proof}
Let $E$ be the subset of $X$ consisting of all points $x$ such that the corresponding card $Y_x=X \setminus \singleton{x}$ has at least three components. Assume for a contradiction that $E$ is uncountable.

Let $D$ be a countable dense subset of $X$. For each point $x$ in $E$, there is a partition of $Y_x$ into disjoint non-empty $Y_x$-clopen sets $C_{x,1}$, $C_{x,2}$ and $C_{x,3}$. and points $d_{x,i} \in D \intersect C_{x,i}$. Since $X$ is $T_1$ and connected, every $C_{x,i}$ is $X$-open, and its closure in $X$ equals $C_{x,i} \union \singleton{x}$, a connected set. In particular, we may pick points $d_{x,i} \in C_{x,i} \cap D$.

By the pigeon-hole principle, there must be distinct points $x$ and $y$ in $E$ such that $\Set{d_{x,1},d_{x,2},d_{x,3}} = \Set{d_{y,1},d_{y,2},d_{y,3}}$. Without loss of generality, we may assume $d_{x,i}=d_{y,i}$, and hence that $C_{x,i}$ intersects $C_{y,i}$ for all $i$. It follows that the connected set $C_{x,1} \union C_{x,2} \union \singleton{x}$ must contain $y$, since a subset of $Y_y$ intersecting both $C_{y,1}$ and $C_{y,2}$ must be disconnected.

But by a symmetric argument, also $C_{x,1} \union C_{x,3} \union \singleton{x}$ and $C_{x,2} \union C_{x,3} \union \singleton{x}$ have to contain $y$, from which we conclude that $x=y$, a contradiction.
\end{proof}

\begin{theorem}\label{thm:SeparableConnectedness}
Being a Tychonoff, separable and connected space is weakly reconstructible.
\end{theorem}
\begin{proof}
First note that being a Tychonoff separable space is reconstructible, and hence also weakly reconstructible. Next, if $X$ has a connected card, then every reconstruction is connected, since $i(X) = 0$ is reconstructible in $T_1$-spaces by Theorem \ref{mythm18}. Hence we may assume that $X$ is a cut-space, i.e.\ that every card is disconnected.
We now claim:

\textit{Claim:} If $X$ is a connected, separable, Tychonoff space without isolated points and without a connected card, then there is a card $X \setminus \Set{x_0}$ having two components $C_1$ and $C_2$ such that both have uncountably many cut-points.

\textit{Proof of claim:} Let $E$ again denote the set of points $x$ such that the corresponding card $Y_x=X \setminus \singleton{x}$ has at least three components. Then $E$ is countable by Theorem \ref{thm:Whyburn}. As $X$ is connected Tychonoff, so uncountable, we can find some $x_0 \not\in E$ such that $Y_{x_0}$ has indeed two components $C_{1}$ and $C_{2}$. Both components are non-trivial connected Tychonoff spaces, so uncountable.

To establish the claim, it suffices to show that all points in $C_{1} \setminus E$ are cut-points of $C_{1}$. Assume for a contradiction that $y \in C_{1} \setminus E$ is not a cut-point of $C_{1}$. As before, the card $Y_y$ has two components $D_{1}$ and $D_{2}$. Since $C_{2}$ is a connected subset of $Y_y$, we may assume without loss of generality that $C_{2} \subset D_{2}$. It follows that closure of ${C_{2}}$ in $X$, equal to $C_{2} \cup \singleton{x_0}$, is also a connected subset of $Y_y$, and hence $C_{2} \cup \singleton{x_0} \subset D_{2}$.

However, the closure of $C_{1} \setminus \Set{y}$ in $X$, a connected set by assumption on $y$, also contains $x_0$. Since it intersects  $D_{2}$ in $x_0$, it must be contained in $D_{2}$. Thus $D_{1}$ is empty, a contradiction. This establishes the claim.

To conclude the proof, let $X \setminus \Set{x_0}$ be the card mentioned in the claim and assume $Z$ is a disconnected reconstruction of $X$. Find $z_0 \in Z$ such that $Z \setminus \singleton{z_0} \cong X \setminus \singleton{x_0}$. It follows that $Z \setminus \singleton{z_0}$ has two components $F_1$ and $F_2$ that correspond to $C_1$ and $C_2$ of $X \setminus \singleton{x_0}$.

Note that without loss of generality, $z_0 \notin \closure{F_2}$, as otherwise $Z$ would be connected. Hence, $F_2$ and $F_1 \cup \singleton{z}$ are clopen subsets of $Z$. But $F_2$ has a uncountably many cut-points $y$ such that $F_2 \setminus \Set{y}$ has  components $C_{y,1}$ and $C_{y,2}$. Thus, for all those $y$, the card $Z \setminus \singleton{y}$ has a partition into three non-empty open sets, namely $F_1 \cup \singleton{z}$, $C_{y,1}$ and $C_{y,2}$. We conclude that the multi-deck of $Z$ contains uncountably many cards having at least three components. Therefore, by Theorem \ref{thm:Whyburn}, $Z$ cannot be a reconstruction of $X$.
\end{proof}

It is easy to see that for spaces having only one card, the notions of reconstruction and weak reconstruction coincide. Hence, we obtain the following corollary. 

\begin{corollary}
\label{mycor1}
Every reconstruction of the reals $\R$ is connected. \qed
\end{corollary}

All of our counterexamples of non-reconstructible spaces in Section \ref{section:counterexamples} rely on the failure of connectedness in that a possible different reconstruction has many components. Given the previous theorem and also Theorem \ref{thm:allReconstructionsCompact} it is therefore interesting to ask:
\begin{question}
	If $X$ is connected and every reconstruction of $X$ is connected, is $X$ reconstructible?

	If not, are separable connected Tychonoff $X$ reconstructible?

	If not, are compact connected metrizable spaces reconstructible?
\end{question}

\providecommand{\bysame}{\leavevmode\hbox to3em{\hrulefill}\thinspace}
\providecommand{\MR}{\relax\ifhmode\unskip\space\fi MR }
\providecommand{\MRhref}[2]{%
  \href{http://www.ams.org/mathscinet-getitem?mr=#1}{#2}
}
\providecommand{\href}[2]{#2}

\end{document}